\documentclass[english,4paper,10pt]{article}
\usepackage{geometry}
\usepackage[utf8]{inputenc}
\usepackage[T1]{fontenc}
\usepackage[english]{babel}
\usepackage[autostyle=true]{csquotes}
\usepackage{latexsym, amssymb, amsmath}
\usepackage{amsthm}
\usepackage{url}
\usepackage{float}
\usepackage{caption}
\usepackage{subcaption}
\usepackage{color}   
\usepackage{hyperref}
\hypersetup{
	colorlinks=true, 
	linkcolor=black,
	filecolor=magenta,      
	urlcolor=cyan,
	linktoc=all,     
}
\usepackage[ruled,vlined]{algorithm2e}
\usepackage{graphicx}
\graphicspath{ {./images/} }
\usepackage{amscd} 

\newcommand{\R}{\mathbb{R}}

\newcommand{\F}{\mathbb{F}}

\newcommand{\mg}{\mathfrak{g}}

\newtheorem{thm}{Theorem}[section]

\newtheorem{definition}[thm]{Definition}
\newtheorem{rem}[thm]{Remark}
\newtheorem{prop}[thm]{Proposition}
\newtheorem{ex}[thm]{Example}

\makeatletter
\def\blfootnote{\xdef\@thefnmark{}\@footnotetext}
\makeatother

\date{}

\usepackage[%
backend=bibtex,%
style=numeric-comp,%
sorting=none,%
hyperref,%
]{biblatex}
\begin{document}
\sloppy

\title{Two-step nilpotent Leibniz algebras \blfootnote{\textit{\textup{2020} Mathematics Subject Classification}: 17A32, 22A30, 20M99.}} \maketitle
\noindent
{{Gianmarco La Rosa}\footnote[1]{Supported by University of Palermo.}, {Manuel Mancini}\footnote[2]{Supported by University of Palermo.} \\
	\footnotesize{Dipartimento di Matematica e Informatica}\\
	\footnotesize{Universit\`a degli Studi di Palermo, Via Archirafi 34, 90123 Palermo, Italy}\\
	\footnotesize{gianmarco.larosa@unipa.it}, ORCID: 0000-0003-1047-5993 \\
	\footnotesize{manuel.mancini@unipa.it}, ORCID: 0000-0003-2142-6193}

\begin{abstract}
In this paper we give a complete classification of two-step nilpotent Leibniz algebras in terms of Kronecker modules associated with pairs of bilinear forms. In particular, we describe the complex and the real case of the indecomposable \emph{Heisenberg Leibniz algebras} as a generalization of the classical $(2n+1)-$dimensional Heisenberg Lie algebra $\mathfrak{h}_{2n+1}$. Then we use the Leibniz algebras - Lie local racks correspondence proposed by S.\ Covez to show that nilpotent real Leibniz algebras have always a global integration. As an application, we integrate the indecomposable nilpotent real Leibniz algebras with one-dimensional commutator ideal. We also show that every Lie quandle integrating a Leibniz algebra is induced by the conjugation of a Lie group and the Leibniz algebra is the Lie algebra of that Lie group.
\end{abstract}

\section*{Introduction}

Leibniz algebras were first introduced by J.-L.\ Loday in \cite{loday1993version} as a non-antisymmetric version of Lie algebras, and many results of Lie algebras were also established in Leibniz algebras. Earlier, such algebraic structures had been considered by A.\ Blokh, who called them D-algebras \cite{blokhLie}. Leibniz algebras play a significant role in different areas of mathematics and physics.

In \cite{falcone2016action} and \cite{bartolone2011nilpotent} Lie algebras with small dimensional commutator ideals have been considered. In this paper our main goal is to classify all the two-step nilpotent Leibniz algebras over a field of characteristic different from $2$.\\ 

The first section is devoted to some background material on Leibniz algebras which will be useful for the rest of the paper. We refer the reader to \cite{ayupov2019leibniz} for more details.

In the second section we use the definition of \emph{Kronecker module} to classify all the indecomposable nilpotent Leibniz algebras with one-dimensional commutator ideal. We associate with each of them a pair of bilinear forms, one symmetric and one skew-symmetric, and we use the simultaneous reduction of this pair to show that there are only three classes of nilpotent Leibniz algebras with one-dimensional commutator ideal. Up to isomorphisms, the first class is determined by the companion matrix of the power of a monic irreducible polynomial, and the other two are unique. In this way, we give the definition of \emph{Heisenberg Leibniz algebra}, \emph{Kronecker Leibniz algebra} and \emph{Dieudonné Leibniz algebra}.

In section 3 we study the complex and the real case of the indecomposable Heisenberg Leibniz algebras. In this case, when the field is $\mathbb{C}$, it is more convenient to use the Jordan canonical form of the companion matrix of the polynomial $(x-a)^k$, that is a $k \times k$ Jordan block of eigenvalue $a$. Moreover, when the dimension is $3$, we determine all the isomorphism classes of these algebras.

In the last section we give a global answer to the \textit{coquecigrue problem} for the nilpotent real Leibniz algebras. We mean the problem, formulated by J.-L.\ Loday \cite{loday1993version}, of finding a generalization of the Lie third theorem, which associates a local Lie group with any (real or complex) Lie algebra, to Leibniz algebras. That is, given any Leibniz algebra $\mg$, one wants to find a manifold endowed with a smooth map which plays the role of $\operatorname{Ad}$ for Lie groups, and such that the tangent space at a distinguished point, endowed with the differential $\operatorname{ad}$ of $\operatorname{Ad}$, is a Leibniz algebra isomorphic to $\mg$. In the special case of (real) Lie algebras $\mg$, one wants to obtain the usual simply connected Lie group associated with $\mg$.

We prove that, if a Lie rack $R$ integrating a left Leibniz algebra $\mg$ is idempotent, then $\mg$ is a Lie algebra and the multiplication of $R$ is induced by the conjugation of a Lie group integrating $\mg$. Moreover, we want to show that the integration of nilpotent Leibniz algebras is global. As an application, we describe the Lie global racks integrating the indecomposable nilpotent Leibniz algebras with one-dimensional commutator ideal classified in the second section.

\section{Preliminaries}

We assume that $\F$ is a field with $char(\F)\neq2$. For the general theory we refer to \cite{ayupov2019leibniz}.

\begin{definition}
	A \emph{left Leibniz algebra} over $\mathbb{F}$ is a vector space $L$ over $\mathbb{F}$ endowed with a bilinear map (called $commutator$ or $bracket$) $\left[-,-\right]\colon L\times L \rightarrow L$ which satisfies the \emph{left Leibniz identity}
	$$
	\left[x,\left[y,z\right]\right]=\left[\left[x,y\right],z\right]+\left[y,\left[x,z\right]\right],\;\;\forall x,y,z\in L.
	$$
	
\end{definition}
	In the same way we can define a right Leibniz algebra, using the right Leibniz identity
	$$
	\left[\left[x,y\right],z\right]=\left[[x,z],y\right]+\left[x,\left[y,z\right]\right],\;\;\forall x,y,z\in L.
	$$
	A Leibniz algebra that is both left and right is called \emph{symmetric Leibniz algebra}. From now on we assume that $\dim_\F L<\infty$.\\
	
	 An equivalent way to define a left Leibniz algebra, is to say that the (left) adjoint map $\operatorname{ad}_x=\left[x,-\right]$ is a derivation, for every $x\in L$. Clearly every Lie algebra is a Leibniz algebra and every Leibniz algebras with skew-symmetric commutator is a Lie algebra. Thus it is defined an adjunction (see \cite{mac2013categories}) between the category \textbf{LieAlg} of the Lie algebras and the category \textbf{LeibAlg} of the Leibniz algebras. The left adjoint of the functor inclusion $i\colon\textbf{LieAlg}\rightarrow \textbf{LeibAlg}$ is the functor $\pi\colon\textbf{LeibAlg} \rightarrow \textbf{LieAlg}$ that associates with every Leibniz algebra $L$ the quotient $L/\operatorname{Leib}(L)$, where $\operatorname{Leib}(L)=\operatorname{Span}_\F\left\{\left[x,x\right] |\,x\in L\right\}$. We observe that $\operatorname{Leib}(L)$ is the smallest bilateral ideal of $L$ such that $L/\operatorname{Leib}(L)$ is a Lie algebra. Moreover $\operatorname{Leib}(L)$ is an abelian algebra.
	 
	  As in the case of Lie algebras, a $derivation$ of a Leibniz algebra is a linear map $d\colon L\rightarrow L$ such that 
	 $$
	 d(\left[x,y\right])=\left[d(x),y\right]+\left[x,d(y)\right]\,\,\forall x,y\in L.
	 $$
	The left multiplications are particular derivations called \emph{inner derivations}. With the usual bracket\\ \hbox{$\left[d_1,d_2\right]=d_1\circ d_2 - d_2\circ d_1$}, the set $\operatorname{Der}(L)$ is a Lie algebra and the set  $\operatorname{Inn}(L)$ of all inner derivations of $L$ is an ideal of $\operatorname{Der}(L)$. Furthermore, $\operatorname{Aut}(L)$ is a Lie group and the associated Lie algebra is $\operatorname{Der}(L)$. \\
	
	We define the left and the right center of a Leibniz algebra
	$$
	\operatorname{Z}_l(L)=\left\{x\in L \,|\,\left[x,L\right]=0\right\},\,\,\, \operatorname{Z}_r(L)=\left\{x\in L \,|\,\left[L,x\right]=0\right\},
	$$
	and we observe that they coincide when $L$ is a Lie algebra. The \emph{center of L} is $\operatorname{Z}(L)=\operatorname{Z}_l(L)\cap \operatorname{Z}_r(L)$. In the case of symmetric Leibniz algebras, the left center and the right center are bilateral ideals, but in general $\operatorname{Z}_l(L)$ is an ideal of the left Leibniz algebra $L$, meanwhile the right center is not even a subalgebra.
	
	\begin{definition}
	Let $L$ be a left Leibniz algebra over $\F$ and let
	$$
	L^{(0)}=L,\,\,L^{(k+1)}=[L,L^{(k)}],\,\,\forall k\geq0,
	$$ be the \emph{lower central series of $L$}. $L$ is \emph{$n-$step nilpotent} if $L^{(n-1)}\neq0$ and $L^{(n)}=0$.
	\end{definition}

One can directly prove the following.

\begin{prop}
If $L$ is a left two-step nilpotent Leibniz algebra, then $L^{(1)}=[L,L]\subseteq \operatorname{Z}(L)$ and $L$ is symmetric.
\end{prop}

\begin{prop}
	If $L$ is a left nilpotent Leibniz algebra with $\dim_\F L^{(1)}=1$, then  $L^{(1)}\subseteq \operatorname{Z}(L)$ and $L$ is symmetric.
\end{prop}
	
	\section{Nilpotent Leibniz algebras with one-dimensional commutator ideal}

Let $L$ be a two-step nilpotent Leibniz algebra, and let $\lbrace z_1,\ldots,z_t\rbrace$ be a basis of the commutator ideal $[L,L]$, thus, for any $x,y\in L$,
$$
[x,y]=\phi_1(x,y)z_1+\cdots+\phi_t(x,y)z_t,
$$
where $\phi_i\colon L\times L\rightarrow \F$ is a bilinear form, for every $i=1,\ldots,t$. In the Lie algebra case, $\phi_1,\ldots,\phi_t$ are skew-symmetric forms.
In this section we reduce such a given bilinear form to a suitable canonical representation.\\

Let $L$ be a nilpotent Leibniz algebra with one-dimensional commutator ideal $\left[L,L\right]=\operatorname{Span}_\F\left\{z\right\}$, where $z \in L$ is fixed, and let $\left[x,y\right]=\phi(x,y)z$, where $\phi:L\times L\rightarrow \F$ is a bilinear form. We observe that, if $L$ is not a Lie algebra, then  $\operatorname{Leib}(L)=[L,L] \subseteq \operatorname{Z}(L)$.\\

	We can decompose $\phi$ into its symmetric and skew-symmetric parts
	$$
	\sigma=\dfrac{\phi+\phi^t}{2},\,\,\,\, \alpha=\dfrac{\phi-\phi^t}{2}.
	$$
	A complete classification of this pair of bilinear forms gives a classification of this class of Leibniz algebras. 
	The simultaneous reduction to canonical form of a pair of bilinear forms was initiated by L.\ Kronecker in \cite{kronecker1890algebraische} and then was studied by J.\ Dieudonné in \cite{Dieudonné}.
	
	\begin{definition} 
	A Kronecker module is a quadruple $\left(V_1,V_2,f_1,f_2\right)$, where $V_1$, $V_2$ are vector spaces over $\F$ and $f_1,f_2\colon V_1\rightarrow V_2$ are linear maps.
	\end{definition}

Every Kronecker module is the  direct sum of indecomposable modules and these were completely classified by several authors (cf. \cite{Falcone2004}). Given a pair of bilinear forms $\alpha,\sigma$ defined over a Leibniz algebra $L$ as above, it is possible to associate the Kronecker module $\left(L,L^*,\bar{\alpha},\bar{\sigma}\right)$, where $L^*$ is the algebraic dual space of $L$ and, for every $x,y$ in $L$, $\bar{\alpha}(x)$ and $\bar{\sigma}(y)$ are defined by
$$
\bar{\alpha}(x)\colon z\mapsto \alpha(x,z),\,\,\,\,\bar{\sigma}(y)\colon z\mapsto \sigma(y,z), \,\, \forall z\in L.
$$ 

We say that $L$ is \emph{decomposable} if $L = U\oplus U^\perp$, where $U^\perp$ is the orthogonal space of $U$ with respect to both $\alpha$ and $\sigma$. In this case it is possible to find a basis of $L$ such that the matrix associated with $\phi$ is a diagonal block matrix. With standard arguments we can assume that $L$ is indecomposable.\\

The indecomposable modules $\left(L,L^*,\bar{\alpha},\bar{\sigma}\right)$ turn out to be one of the following three pairs

\begin{equation}\label{modulo1}
 \begin{pmatrix}
 	0 & I_n \\ -I_n & 0
 \end{pmatrix},
\begin{pmatrix}
	0 & A \\ A^t & 0  
\end{pmatrix}
\end{equation}
\begin{equation}\label{modulo2}
\begin{pmatrix}
	0 & A \\ -A^t & 0
\end{pmatrix},
\begin{pmatrix}
	0 & I_n \\ I_n & 0  
\end{pmatrix}
\end{equation}
\begin{equation}\label{modulo3}
\begin{pmatrix}
	0 & J_1 \\ -J_1^t & 0
\end{pmatrix},
\begin{pmatrix}
	0 & J_2 \\ J_2^t & 0  
\end{pmatrix}
\end{equation}
where $A\in \operatorname{M}_n(\F)$, and $J_1,J_2$ are the matrices associated with linear applications $F_1,F_2\colon \F^n\rightarrow \F^{n+1}$ defined by $F_1(x_1,\ldots,x_n)=(x_1,\ldots,x_n,0)$ and $F_2(x_1,\ldots,x_n)=(0,x_1,\ldots,x_n)$.\\ \\
The following result is a useful tool for the classification of the indecomposable nilpotent Leibniz algebras with one-dimensional commutator ideals. .

\begin{prop}
	Let $L_1$ and $L_2$ be Leibniz algebras of dimension $n$ with one-dimensional commutator ideals $\left[L_1,L_1\right]=\mathbb{F}z_1$ and $\left[L_2,L_2\right]=\mathbb{F}z_2$, and let $\phi_1$ and $\phi_2$ be the bilinear forms associated with $L_1$ and $L_2$ respectively. Let $\Phi_1$ and $\Phi_2$ be the matrices of $\phi_1$ and $\phi_2$ respectively. $L_1$ is isomorphic to $L_2$ if and only if $\Phi_1$ is congruent to $\Phi_2$.
\end{prop}
\begin{proof}
	Let $\varphi \colon L_1\rightarrow L_2$ be a Leibniz algebras isomorphism and let $\left\{e_1,\ldots,e_{n-1},z_1\right\}$ be a basis of $L_1$. Then $\varphi(z_1)=kz_2$, for some $k \in \mathbb{F}^*$, and $\left\{\varphi(e_1),\ldots,\varphi(e_{n-1}),kz_2\right\}$ is a basis of $L_2$ such that the associated matrix is $\Phi_1$. Then there exists a matrix $P\in \operatorname{GL}_n(\F)$ such that $P\Phi_2 P^t=\Phi_1$.\\
	Conversely, we suppose that there exists $P\in \operatorname{GL}_n(\F)$ such that $P\Phi_1 P^t=\Phi_2$. $\Phi_1$ and $\Phi_2$ are matrices associated with bilinear forms, so $P$ induces a change of basis $\left\{e_1,\ldots,e_{n-1},z_1\right\}\rightarrow\left\{\overline{e_1},\ldots,\overline{e_{n-1}},kz_1\right\}$, with $k \in \mathbb{F}^*$, of $L_1$. Thus the isomorphism between $L_1$ and $L_2$ is given by the linear map 
	$$
	\varphi(\overline{e_i})=e_i',\,\, \varphi(kz_1)=z_2, 
	$$
	where $\left\{e_1',\ldots,e_{n-1}',z_2\right\}$ is a basis of $L_2$.
\end{proof}
We observe that, if $X\in \operatorname{GL}_n(\F)$, then the matrix
$$
\begin{pmatrix}
	X & 0 \\
	0 & \left(X^{-1}\right)^t
\end{pmatrix}\
$$
 induces a change of basis that transforms the canonical pairs (\ref{modulo1}) and (\ref{modulo2}) respectively in
\begin{equation*}
 \left(\begin{pmatrix}
	0 & I_n \\ -I_n & 0
\end{pmatrix},
\begin{pmatrix}
	0 & XAX^{-1} \\ \left(XAX^{-1}\right)^t & 0  
\end{pmatrix}\right),
\left(\begin{pmatrix}
0 & XAX^{-1} \\ -\left(XAX^{-1}\right)^t & 0 
\end{pmatrix},
\begin{pmatrix}
	 	0 & I_n \\ I_n & 0
\end{pmatrix}\right),
\end{equation*}
so we can always reduce $A$ in its rational canonical form. Thus, with the assumption that $L$ is indecomposable, if $A\in \operatorname{M}_n(\F)$ is not the zero matrix, then $A$ is the companion matrix of a power of a monic irreducible polynomial $f(x)\in\F\left[x\right]$ (see \cite{jacobson2012basic}).

Moreover, if $A$ is not a singular matrix, then the second canonical pair is equivalent to the first one, in fact
\begin{align*}
&\begin{pmatrix}
	A^{-1} & 0 \\
	0 & I_n
\end{pmatrix}\begin{pmatrix}
0 & A \\ -A^t & 0  
\end{pmatrix}\begin{pmatrix}
A^{-1} & 0 \\
0 & I_n
\end{pmatrix}^t=\begin{pmatrix}
0 & I_n \\
-I_n & 0
\end{pmatrix}\\
&\begin{pmatrix}
	A^{-1} & 0 \\
	0 & I_n
\end{pmatrix}\begin{pmatrix}
	0 & I_n \\ I_n & 0
\end{pmatrix}\begin{pmatrix}
	A^{-1} & 0 \\
	0 & I_n
\end{pmatrix}^t=\begin{pmatrix}
	0 & A^{-1} \\
	\left(A^{-1}\right)^t & 0
\end{pmatrix}.
\end{align*}
Otherwise, we can represent a singular matrix as an $n \times n$ Jordan block of eigenvalue zero and we have a unique Kronecker module of type $(\ref{modulo2})$ up to isomorphisms.

\begin{definition}
Let $f(x)\in\F\left[x\right]$ be a monic irreducible polynomial. Let $k\in\mathbb{N}$ and let $A$ be the companion matrix of $f(x)^k$. We define the \emph{Heisenberg} Leibniz algebra $\mathfrak{l}_{2n+1}^A$ as the $(2n+1)$-dimensional indecomposable Leibniz algebra with associated Kronecker module of type $(\ref{modulo1})$.
\end{definition}

In general, if $A=(a_{ij})\in \operatorname{M}_n(\F)$ and if we fix a basis $\left\{e_1,\ldots,e_n,f_1,\ldots,f_n,h\right\}$ of $\mathfrak{l}_{2n+1}^{A}$, then the non-trivial commutators are
$$
\left[e_i,f_j\right]=\left(\delta_{ij}+a_{ij}\right)h,\,\, \left[f_j,e_i\right]=\left(-\delta_{ij}+a_{ij}\right)h,\,\, \forall i,j=1,\ldots,n,
$$
so we can associate with $\mathfrak{l}_{2n+1}^{A}$ the following structure matrix
$$
	\left(
	\begin{array}{c|c|c}
		&  & 0 \\
		0 & I_n+A & \vdots \\
		& & 0 \\
		\hline
		&  & 0 \\
		-I_n+A^t & 0 & \vdots \\
		& & 0 \\
		\hline
		0 \cdots 0 & 0 \cdots 0 & 0
	\end{array}
	\right).
$$

Notice that, if $\left(a_{ij}\right)$ is the zero matrix, then we obtain the classical Heisenberg algebra $\mathfrak{h}_{2n+1}$. 

\begin{definition}
	Let $n \in \mathbb{N}$ and let $A$ be the companion matrix of the polynomial $x^n$. We define the \emph{Kronecker} Leibniz algebra $\mathfrak{k}_{n}$ as the $(2n+1)$-dimensional indecomposable Leibniz algebra with associated Kronecker module of type $(\ref{modulo2})$.
\end{definition}

\begin{definition}
We define the \emph{Dieudonné} Leibniz algebra $\mathfrak{d}_n$ is the $(2n+2)-$dimensional Leibniz algebra with associated Kronecker module of type $(\ref{modulo3})$.
\end{definition}

We observe that, for every $n\in\mathbb{N}$, the Kronecker Leibniz algebra $\mathfrak{k}_{n}$ and the Dieudonné Leibniz algebra $\mathfrak{d}_{n}$ are not Lie algebras and they are unique up to isomorphisms, because of the unicity of the Kronecker modules of type $(\ref{modulo2})$ and $(\ref{modulo3})$.

\section{Complex and Real Heisenberg Leibniz algebras}

Now we want to describe in details the indecomposable Heisenberg Leibniz algebras in the case the field $\F$ is $\mathbb{C}$ or $\R$.

\subsection{The case $\F=\mathbb{C}$}
Let $k\in\mathbb{N}$ and let $f(x)=x-a\in\mathbb{C}[x]$. Then the companion matrix of $f(x)^k$ is
$$
A=
\begin{pmatrix}
	0 & \cdots & \cdots & 0 & -c_k \\
	1 & \ddots & & \vdots & -c_{k-1} \\
	0 & 1 & \ddots & \vdots & \vdots \\
	\vdots & \ddots & \ddots & 0 & \vdots \\
	0 & \cdots & 0 & 1 & -c_1
\end{pmatrix}\in \operatorname{M}_k(\mathbb{C}),
$$
where  $c_j=\binom{k}{j-1}(-a)^{k-j+1}$, for every $j=1,\ldots,k$. In this case, however, it is more convenient to use the Jordan canonical form. Indeed, it is well known that the matrix $A$ is similar to the $k\times k$ Jordan block of eigenvalue $a$
$$
J_a=\begin{pmatrix}
	a&0&0&\cdots&0         \\                       
	1&\ddots&\ddots&&\vdots      \\
	\vdots&\ddots&\ddots&\ddots&0\\
	\vdots&&\ddots&\ddots&0     \\
	0&\cdots&\cdots&1&a   
\end{pmatrix}.
$$
Thus $\mathfrak{l}_{2k+1}^A\cong\mathfrak{l}_{2k+1}^{J_a}$ and the Leibniz bracket is given by
\begin{align*}
	\left[e_1,f_j\right]=& \delta_{1,j}(1+a)h \\
\left[e_i,f_j\right]=&(\delta_{i,j}(1+a)+\delta_{i-1,j})h,\,\,\forall i=2,\ldots,k; \\
\left[f_j,e_i\right]=&(\delta_{i,j}(-1+a)+\delta_{i,j+1})h,\,\,\forall j=1,\ldots,k-1\\
\left[f_k,e_i\right]=& \delta_{i,k}(-1+a)h,
\end{align*} 
where $\left\{e_1,\ldots,e_n,f_1,\ldots, f_n,h\right\}$ is a basis of $\mathfrak{l}_{2k+1}^{J_a}$.

\begin{prop}
	Let $a \in \mathbb{C}$. The Heiseberg-Leibniz algebras $\mathfrak{l}_{2k+1}^{J_a}$ and $\mathfrak{l}_{2k+1}^{J_{-a}}$ are isomorphic.
\end{prop}

\begin{proof} The algebras
	$\mathfrak{l}_{2k+1}^{J_a}$ and $\mathfrak{l}_{2k+1}^{-J_a^t}$ are isomorphic via the linear map $\varphi$ defined by
	$$
	\varphi(e_i)=f_i', \; \; \varphi(f_i)=e_i', \; \; \varphi(h)=-h', \; \; \forall i=1,\cdots,n
	$$
	where $\lbrace e_1,\cdots,e_n,f_1,\cdots,f_n,h\rbrace$ and $\lbrace e_1',\cdots,e_n',f_1',\cdots,f_n',h'\rbrace$ are bases of $\mathfrak{l}_{2k+1}^{J_a}$ and $\mathfrak{l}_{2k+1}^{-J_a^t}$ respectively. Moreover the matrix $-J_a^t$ is similar to the $n \times n$ Jordan block $J_{-a}$. Thus $\mathfrak{l}_{2k+1}^{J_a} \cong \mathfrak{l}_{2k+1}^{J_{-a}}$.
\end{proof}

When $k=1$ the converse result is also true.

\begin{prop}
	Let $a,a'\in\mathbb{C}$. The Heisenberg Leibniz algebras $\mathfrak{l}_3^a$ and $\mathfrak{l}_3^{a'}$ are isomorphic if and only if $a'=\pm a$.
\end{prop}

\begin{proof}
	It is easy to check that the matrix
	 $$
	\begin{pmatrix}
		0 & 1 & 0 \\
		1 & 0 & 0 \\
		0 & 0 & -1 \\
	\end{pmatrix}
	$$
	defines a Leibniz algebras isomorphism between $\mathfrak{l}_3^a$ and $\mathfrak{l}_3^{-a}$. Conversely, let $\varphi \colon \mathfrak{l}_3^a\rightarrow \mathfrak{l}_3^{a'}$ be a Leibniz algebra isomorphism defined by 
	$$
	\varphi(x)=\alpha x'+\beta y'+\gamma z',\,\, \varphi(y)=\alpha' x'+\beta' y'+\gamma' z',\,\, \varphi(z)=kz',
	$$
	where $\left\{x,y,z\right\}$ and $\left\{x',y',z'\right\}$ are basis of $\mathfrak{l}_3^a$ and $\mathfrak{l}_3^{a'}$ respectively. Thus
	\begin{equation*}
		0=\varphi(\left[x,x\right])=\left[\varphi(x),\varphi(x)\right]=\alpha\beta\left( 1+a'-1+
		a' \right)z'=2\alpha\beta a'z',
	\end{equation*}
	
	\begin{equation*}
		0=\varphi(\left[y,y\right])=\left[\varphi(y),\varphi(y)\right]=\alpha'\beta'\left( 1+a'-1+
		a' \right)z'=2\alpha'\beta' a'z',
	\end{equation*}
	
	\begin{equation*}
		k\left(1+a\right)z'=\varphi(\left[x,y\right])=\left[\varphi(x),\varphi(y)\right]=\left( \alpha'\beta\left(1+a'\right)+\alpha\beta'\left(-1+a'\right)\right)z',
	\end{equation*}
	
	\begin{equation*}
		k\left(-1+a\right)z'=\varphi(\left[y,x\right])=\left[\varphi(y),\varphi(x)\right]=\left( \alpha\beta'\left(1+a'\right)+\alpha'\beta\left(-1+a'\right)\right)z'.
	\end{equation*}
	We have that $\left(\alpha',\beta\right)=(0,0)$ or $\left(\alpha,\beta'\right)=(0,0)$. In the first case $\varphi$ is the identity map and $a=a'$. In the second case $\varphi$ is defined by $\varphi(x)=y',\,\, \varphi(y)=x'$ and $\varphi(z)=-z'$, thus $a'=-a$.
\end{proof}

\subsection{The case $\F=\mathbb{R}$}
Irreducible polynomials in $\R[x]$ have degree one or two. Let $f(x)\in\R[x]$ be an irreducible monic polynomial. If $f(x)=x-a$, then we obtain the same results of the previous case. So we suppose that $f(x)=x^2+bx+c$, with $b^2-4c<0$.\\
\\
Let $z=\alpha + i \beta \in \mathbb{C}$ be a root of $f(x)$. Then $f(x)=(x-z)(x-\bar{z})$ and the companion matrix $A$ of $f(x)^k$ in similar to the $2k \times 2k$ real block matrix
$$
J_R=\begin{pmatrix}
	R&0&\cdots&0\\I_2&R&\cdots&0\\\vdots&\ddots&\ddots&\vdots\\0&\cdots&I_2&R
\end{pmatrix},
$$
where
$$
R=R_{\alpha,\beta}=\begin{pmatrix}
	\alpha & \beta \\
	-\beta & \alpha
\end{pmatrix}
$$
is the realification of the complex number $z$. Thus $\mathfrak{l}_{4k+1}^{A} \cong \mathfrak{l}_{4k+1}^{J_R}$ and the structure matrix is given by
\begin{equation}
	$$\[
	\left(
	\begin{array}{c|c|c}
		&  & 0 \\
		0 & I_n+J_R & \vdots \\
		& & 0 \\
		\hline
		&  & 0 \\
		-I_n+J_R^t & 0 & \vdots \\
		& & 0 \\
		\hline
		0 \cdots 0 & 0 \cdots 0 & 0
	\end{array}
	\right).
	\]$$
\end{equation}
In the case that $k=1$, the real Heisenberg Leibniz algebra $\mathfrak{l}_{5}^R$ is the realification of the complex algebra $\mathfrak{l}_3^z$. Thus we can conclude that
\begin{prop}
	Let $f(x),g(x) \in \mathbb{R}[x]$ be two irreducible monic polynomials of degree two and let $z,z' \in \mathbb{C}$ be roots of $f(x)$ and $g(x)$ respectively. Let $R,R' \in \operatorname{M}_2(\mathbb{R})$ be the realification of the complex numbers $z$ and $z'$. Then $\mathfrak{l}_5^{R} \cong \mathfrak{l}_5^{R'}$ if and only if $R'=\pm R$.
\end{prop}

\begin{proof}
 The algebras $\mathfrak{l}_5^{R}$ and $\mathfrak{l}_5^{R'}$ are the realification of the complex Heisenberg Leibniz algebras $\mathfrak{l}_3^{z}$ and $\mathfrak{l}_3^{z'}$ respectively. From \textit{Proposition 3.2} we know that $\mathfrak{l}_3^{z} \cong \mathfrak{l}_3^{z'}$ if and only if $z = \pm z'$. Moreover, these are $\mathbb{R}$-linear isomorphisms because the matrix associated with the isomorphism $\varphi\colon \mathfrak{l}_3^z \leftrightarrows \mathfrak{l}_3^{-z}$  is the rotation 
 $$
 \begin{pmatrix}
 	0 & 1 & 0 \\
 	1 & 0 & 0 \\
 	0 & 0 & -1 \\
 \end{pmatrix} \in \operatorname{SO}(3).
 $$
  Thus $\mathfrak{l}_5^{R} \cong \mathfrak{l}_5^{R'}$ if and only if $R = \pm R'$.
\end{proof}

\section{Global integration of nilpotent Leibniz algebras}

In the case of the correspondence between a Lie group $G$ and its Lie algebra $\mg=\operatorname{T}_1G$, the bracket 
$$
[x,y]=\operatorname{ad}_x(y),
$$ where
$\operatorname{ad}\colon\mg\rightarrow\mathfrak{gl}(\mg)$, is the differential of the adjoint representation $\operatorname{Ad}\colon G\rightarrow \operatorname{GL}(\mg)$, which in turn is the differential of the conjugation map $\gamma: x \mapsto \gamma_x$, where $\gamma_x(y):=x \rhd y$ and $x \rhd y = xyx^{-1}$, for every $x,y \in G$. In the case of Leibniz algebras, there is no hope, as we will see, of finding such a map $\operatorname{Ad}$, but it is still possible to define an algebraic structure $(X,\rhd)$, called a \emph{rack}, whose operation, differentiated twice, defines on $\operatorname{T}_1X$ a Leibniz algebra structure.

From now on, unless explicitly stated, the underlying field of any vector space will be the real numbers.

\begin{definition}
	A (left) rack is a set $X$ with a binary operation $\rhd: X\times X\rightarrow X$ which is (left) autodistributive, that is, for all $x,y,z\in X$, $x\rhd\left(y\rhd z\right)=\left(x\rhd y\right)\rhd\left(x\rhd z\right)$ and such that $x\rhd -:X\times X\rightarrow X$ is a bijection for all $x\in X$. A rack is said to be pointed if there exists an element $1\in X$, called the unit, such that $1\rhd x=x$ and $x\rhd 1=1$ for all $x\in X$. 	A rack $(X,\rhd)$ is a quandle if $x\rhd x=x$, for every $x\in X$ (i.e.\ $\rhd$ is idempotent).
\end{definition}

A \textit{pointed rack homomorphism} is a map $f:X\rightarrow Y$ such that $f(x\rhd y)=f(x)\rhd f(y)$, for all $x,y\in X$ and such that $f(1_X)=1_Y$. 

Every group endowed with the conjugation is a pointed rack, so there is a functor $\operatorname{Conj}\colon \textbf{Grp}\rightarrow\textbf{Rack}$, between the category \textbf{Grp} of groups and the category \textbf{Rack} of racks. This functor has a left adjoint \hbox{$\operatorname{As}\colon \textbf{Rack}\rightarrow\textbf{Grp}$} defined by 
$$
\operatorname{As}(X)=\operatorname{F}(X)/\overline{\langle \lbrace xyx^{-1} \left(x\rhd y^{-1}\right) \;\vert\; xy\in X\rbrace \rangle},
$$ where $\operatorname{F}(X)$ is the free group generated by $X$. 
\begin{definition}
	A Lie rack is a pointed rack $\left(X,\rhd,1\right)$ such that $X$ is a smooth manifold, $\rhd$ is a smooth map and such that for all $x\in X$ $x\rhd -$ is a diffeomorphism.
\end{definition}

In \cite{kinyon2004leibniz} \mbox{M.\ K.\ Kinyon} shows that the tangent space at the unit element of a \textit{Lie rack} $(X,\rhd)$, endowed with the bracket 
$$
[x,y]=\operatorname{ad}_x(y),
$$
where $\operatorname{ad}\colon \operatorname{T}_1X\rightarrow\mathfrak{gl}(\operatorname{T}_1X)$ is the differential of the map $\Phi\colon X \mapsto \operatorname{GL}(\operatorname{T}_1X)$, where $\Phi(x)=\operatorname{T}_1\phi(x)$ and $\phi(x)=x \rhd -$, for every $x \in X$, is a Leibniz algebra. Summarizing we have

\begin{prop}{\cite{kinyon2004leibniz}}
	If $X$ is a Lie rack, then the above bracket $[x,y]=\operatorname{ad}_x(y)$ defines on $\operatorname{T}_1X$ a Leibniz algebra structure.
\end{prop}

The converse problem, that is to find a manifold
endowed with a smooth operation such that the tangent space at the distinguished point, endowed with the differential of this operation, gives a Leibniz algebra isomorphic to the given one, is known as the coquecigrue problem. In \cite{kinyon2004leibniz} M.\ K.\ Kinyon solves the \emph{coquecigrue problem} for the class of \emph{split Leibniz algebras}, that are Leibniz algebras $\mg$ with a bilateral ideal \mbox{$\operatorname{Leib}(\mg) \subseteq \mathfrak{a} \subseteq \operatorname{Z}_l(\mg)$} and with a Lie subalgebra $\mathfrak{h} \subseteq \mg$ such that $\mg=\mathfrak{h} \oplus \mathfrak{a}$ (as a direct sum of vector spaces) and 
$$
[(x,a),(y,b)]=([x,y],\rho_x(b)), \; \; \forall (x,a),(y,b) \in \mathfrak{h} \oplus \mathfrak{a}
$$
where $\rho\colon \mathfrak{h} \times \mathfrak{a} \rightarrow \mathfrak{a}$ is the action on the $\mathfrak{h}$-module $\mathfrak{a}$. More precisely, we have the following.

\begin{prop}{\cite{kinyon2004leibniz}}
	Let $\mg$ be a split Leibniz algebra. Then there exists a Lie rack $X$ such that $\operatorname{T}_1X$ is isomorphic to $\mg$.
\end{prop}

More recently S.\ Covez in \cite{covez2013local} gives a solution to this problem which in general is only local: he shows how to integrate every Leibniz algebra into a local Lie rack. The central point of his result is to see every Leibniz algebra $\mg$ as an abelian extension of his left center $\operatorname{Z}_l(\mg)$ and to integrate explicitly the corresponding Leibniz algebra $2$-cocycle into a Lie local rack $2$-cocycle. However M.\ Bordemann and F.\ Wagemann (see \cite{integrBordemann}) and J.\ Mostovoy (see \cite{acommentMostovoy}) give independently two different answers to the general coquecigrue problem: M.\ Bordemann and F.\ Wagemann's solution is not functorial (as well as S.\ Covez's method, because the left center does not depend in a functorial manner on the Leibniz algebra in general); J.\ Mostovoy's solution is global but does not generalize the classical Lie solution. The general coquecigrue problem is still open.

The aim of this section is to use the Leibniz algebras - Lie local racks correspondence proposed by S. Covez to show that the integration of the two-step nilpotent Leibniz algebras is global. \\

In \cite{covez2013local} S.\ Covez gives the definition of \textit{smooth rack modules}, \textit{rack cohomology} and \textit{cohomolgy theory} for Leibniz algebras. In particular, for $X$ a Lie rack and $A$ a smooth $X-$module, he defines a cochain complex $\lbrace \operatorname{CR}^n\left(X,A\right), d^n_R\rbrace_{n\in\mathbb{N}}$  by setting
$$
\operatorname{CR}^n\left(X,A\right)=\left\{f:X^n\rightarrow A\;\vert\; f(x_1,\ldots,1,\ldots,x_n)=0, \text{$f$ is smooth in a neighborhood of }\left(1,\ldots,1\right)\in X^n\right\}
$$
and $d^n_R$ is the differential operator. Moreover, for $\mg$ a left Leibniz algebra and $M$ a $\mg$-module, he defines a cochain complex $\left\{\operatorname{CL}^n(\mg,M),dL^n\right\}_{n\in\mathbb{N}}$ by setting
$$
\operatorname{CL}^n(\mg,M)=Hom(\mg^{\otimes n},M)
$$
and $dL^n$ is the differential operator.\\

Given a left Leibniz algebra $\mathfrak{g}$, there are several ways to see $\mg$ as an abelian extension of a Lie subalgebra $\mg_0\subseteq \mathfrak{gl}(V)$ by a $\mg_0-$module $\mathfrak{a}$. For example we can take $\mathfrak{a}=\operatorname{Z}_l(\mg)$ and $\mg_0=\mg/\operatorname{Z}_l(\mg)$. Thus we can associate with $\mg$ a short exact sequence
$$
0\rightarrow \operatorname{Z}_l(\mg) \hookrightarrow \mg \twoheadrightarrow \mg_0 \rightarrow 0.
$$
in the category \textbf{LeibAlg}. $\operatorname{Z}_l(\mg)$ is a $\mg_0-$module (in the sense of Lie algebras), so there is a Leibniz algebras $2-$cocycle $\omega\in \operatorname{ZL}^2\left(\mg_0,\operatorname{Z}_l(\mg)\right)$ such that $\mg=\mg_0\oplus_\omega \operatorname{Z}_l(\mg)$. The Leibniz bracket in $\mg$ can be written as follows
$$
\left[\left(x,a\right),\left(y,b\right)\right]=\left(\left[x,y\right]_{\mg_0}, \rho_x(b)+\omega(x,y)\right),
$$
where $\rho\colon \mg_0\times \operatorname{Z}_l(\mg)\rightarrow \operatorname{Z}_l(\mg)$ is the action induced by the $\mg_0$-module structure of $\operatorname{Z}_l(\mg)$. 
\begin{thm}{\cite{covez2013local}}\label{thmcovez}
Every Leibniz algebra $\mg=\mg_0\oplus_\omega \mathfrak{a}$ can be integrated into a local Lie rack of the form 
$$
G_0\times_f \mathfrak{a}
$$
with operation defined by
$$
(g,a)\rhd (h,b)=\left(ghg^{-1}, \phi_g(b)+f(g,h)\right)
$$
and unit $(1,0)$, where $G_0$ is a Lie group such that $\operatorname{Lie}(G_0)=\mg_0$, $\phi$ is the exponentiation of the action $\rho$, \\$f\colon G_0\times G_0\rightarrow \mathfrak{a}$ is the Lie local racks $2-cocycle$ defined by 
$$
f(g,h)=\int_{\gamma_h}\left(\int_{\gamma_g}\tau^2(\omega)^{eq}\right)^{eq},\;\;\forall g,h\in G_0
$$
and $\tau^2(\omega)\in \operatorname{ZL}^1(\mg_0,Hom(\mg_0,\mathfrak{a}))$ is defined by $\tau^2(\omega)(x)(y)=\omega(x,y)$, for all $x,y\in\mg_0$.
\end{thm}

We finally can answer the question whether a Lie rack integrating a Leibniz algebra can be the quandle   $\operatorname{Conj}(G)$, for a suitable Lie group $G$. The answer is no in general, as the following theorem shows.

\begin{thm}\label{thm1}
	Let $R$ be a Lie rack integrating a Leibniz algebra $\mg$. $R$ is a quandle if and only if $\mg$ is a Lie algebra. In particular $R=\operatorname{Conj}(G)$, where $\operatorname{Lie}(G)=\mg$.
\end{thm}

\begin{proof} 
	If $\mg$ is a Lie algebra, then it is clear that $R=Conj(G)$, where $\operatorname{Lie}(G)=\mg$.
	Conversely, we suppose that $R$ is a Lie quandle. Again we can write $\mg=\mg_0\oplus_\omega \operatorname{Z}_l(\mg)$, thus $R$ is of the form $G_0\times_f \operatorname{Z}_l(\mg)$, with multiplication 
	$$
	(g,a)\rhd (h,b)=\left(ghg^{-1}, \phi_g(b)+f(g,h)\right),
	$$
	where $f$ is the Lie racks $2-$cocycle integrating $\omega$. To prove that $\mg$ is a Lie algebra, we have to show that $\left[(x,a),(x,a)\right]=(0,0)$, for all $(x,a)\in\mg$.
	
	The condition $(g,a)\rhd(g,a)=(g,a)$ implies that $f(g,g)=0$, for all $g\in G_0$, and then $\phi_g(a)=a$, for all $a\in Z_l(\mg)$. Indeed the action $\rho$ of $\mg_0$ on $\operatorname{Z}_l(\mg)$ is trivial and $\omega(x,x)=0$ for all $x\in\mg_0$. Finally $R=\operatorname{Conj}(G)$, where $G=G_0\times_F \operatorname{Z}_l(\mg)$ is the Lie group with operation 
	$$
	(g,a)(h,b)=(gh, a+b+F(g,h)),
	$$
	and $F\colon G_0\times G_0\rightarrow \operatorname{Z}_l(\mg)$ is a Lie group $2-$cocycle such that 
	$$
	f(g,h)=F(g,h)-F(g,g^{-1})+F(gh,g^{-1})\;\;\;\forall g,h\in G_0.
	$$
	In fact with this condition we have that
	$$
	(g,a)\rhd (h,b)=(g,a)(h,b)(g,a)^{-1} \;\;\;\forall (g,a),(h,b)\in G_0\times \operatorname{Z}_l(\mg)
	$$
	and the Lie algebra of the Lie group $G$ is clearly $\mg$.
\end{proof}

Now we will claim a result about the integration of  nilpotent Leibniz algebras. We will show that, for this class of Leibniz algebras, the integration proposed by Covez is global. 

\begin{thm}\label{thmnil}
	Every nilpotent  real Leibniz algebra $\mg$ has a global integration into a Lie rack.
\end{thm}
\begin{proof}

Let $\mg$ be a nilpotent real Leibniz algebra and we see $\mg$ as the abelian extension of $\mg_0=\mg/\operatorname{Z}_l(\mg)$ by its left center $\operatorname{Z}_l(\mg)$. Then $\mg_0$ is a nilpotent Lie algebra, thus for every $x\in\mg_0$ the action $\rho_x$ defined by 
$$
\rho_x(a)=\operatorname{ad}_x(a),\;\;\forall a\in \operatorname{Z}_l(\mg),
$$
can be represented as $m\times m$ strictly lower triangular matrix (see \cite{erdmann2006introduction}), where $m=\dim_\R \operatorname{Z}_l(\mg)$. If $G_0$ is the simply connected Lie group integrating $\mg_0$, the action $\phi_g\in\operatorname{Aut}(\operatorname{Z}_l(\mg))$ obtained by the exponentiation of the matrix $\rho_x$, is a unitriangular matrix which entries are polynomial expressions of the coordinates of the vector $x\in\mg_0$, with respect to a fixed basis. Thus, for every $g,h\in G_0$, fixed the two smooth paths $\gamma_g(s)=g^s$ and $\gamma_h(t)=h^t$ in $G_0$ from $1$ to $g$ and from $1$ to $h$ respectively, we have that the Lie racks $2-$cocycle
$$
f(g,h)=\int_{\gamma_h}\left(\int_{\gamma_g}\tau^2(\omega)^{eq}\right)^{eq},\;\;\forall g,h\in G_0
$$
is always defined because it involves the integration of matrices with entries in $\mathbb{R}[s]$ and $\mathbb{R}[t]$. Thus \hbox{$G_0=\mg_0\times_f Z_l(\mg)$} has a Lie global rack structure integrating the nilpotent Leibniz algebra $\mg$.
\end{proof}

In the case that $\mg$ is a two-step nilpotent Leibniz algebra, a Lie rack integrating $\mg$ can be defined without integrating the Leibniz algebras $2-$cocycle associated with $\mg$. 

\begin{thm} \label{thm2}
	Let $(\mg,[-,-])$ be a two-step nilpotent Leibniz algebra and let $\omega\colon \mg_0 \times \mg_0 \rightarrow \left[\mg,\mg\right]$, where $\mg_0=\mg / \left[\mg,\mg\right]$, be the Leibniz algebras $2-$cocycle associated with the short exact sequence
	$$
	0\rightarrow \left[\mg,\mg\right] \hookrightarrow \mg \twoheadrightarrow \mg_0 \rightarrow 0.
	$$
	Then the multiplication
	$$
	\left(x,a\right)\rhd\left(y,b\right)=\left(y,b+\omega(x,y)\right),\;\;\forall \left(x,a\right),\left(y,b\right)\in \mg_0\times\left[\mg,\mg\right]
	$$
	defines a Lie global rack structure on $\mg_0\times \left[\mg,\mg\right]$, such that $\operatorname{T}_{(0,0)}(\mg_0\times_{\omega} \left[\mg,\mg\right], \rhd)$ is a Leibniz algebra isomorphic to $\mg$.
\end{thm}

\begin{proof}
We have $\left[\mg,\mg\right]\subseteq \operatorname{Z}(\mg)$, so we can see $\mg$ as an abelian extension of $\left[\mg,\mg\right]$ by the quotient $\mg_0=\mg/\left[\mg,\mg\right]$ via a Leibniz algebras $2-$cocycle $\omega\in \operatorname{ZL}^2(\mg_0,\left[\mg,\mg\right])$. Thus $\mg=\mg_0\oplus_\omega \left[\mg,\mg\right]$ with bracket 
$$
\left[\left(x,a\right),\left(y,b\right)\right]=\left(0,\omega(x,y)\right).
$$ 
In fact the condition $\left[\mg,\mg\right]\subseteq \operatorname{Z}_l(\mg) \cap \operatorname{Z}_r(\mg)$ implies that the action of $\mg_0$ on $\left[\mg,\mg\right]$ is trivial.
Moreover $\mg_0$ is an abelian algebra, thus a Lie group integrating $\mg_0$ is $G_0=\mg_0$. Then we can define a Lie rack structure on the cartesian product $\mg_0\times\left[\mg,\mg\right]$ by setting 
$$
\left(x,a\right)\rhd\left(y,b\right)=\left(y,b+\omega(x,y)\right)\;\;\forall \left(x,a\right),\left(y,b\right)\in \mg_0\times\left[\mg,\mg\right],
$$
with unit element $\left(0,0\right)$. Finally the tangent space $\operatorname{T}_{\left(0,0\right)}\left(\mg_0\times\left[\mg,\mg\right]\right)$ has a Leibniz algebra structure isomorphic to $\mg$. In fact
\begin{equation*}
\frac{\partial^2}{\partial s\partial t}\bigg|_{s,t=0}(sx,sa)\rhd (ty,tb)=\frac{\partial^2}{\partial s\partial t}\bigg|_{s,t=0}\left(ty,tb+\omega\left(sx,ty\right)\right)=\left(0,\omega\left(x,y\right)\right)=\left[\left(x,a\right),\left(y,b\right)\right].
\end{equation*}
\end{proof}

\begin{rem}
The effective strategy in the proof of \textit{Theorem} \ref{thm2} was to choose $G_0=\mg_0$ as a Lie group integrating the abelian algebra $\mg_0$. If we change the Lie group $G_0$, then the integration may not be global, as the following example illustrates.
\end{rem}

\begin{ex}
	Let $a \in \mathbb{R}$ and let $\mg=\mathfrak{l}_3^a$ be the three-dimensional Heinseberg Lebiniz algebra. Then $[\mg,\mg]=\operatorname{Z}(\mg)\cong\mathbb{R}$ and we can see $\mg$ as an abelian extension of the Lie algebra $\mg_0=\mg / [\mg,\mg] \cong \mathbb{R}^2$ by $\mathbb{R}$. The corresponding Leibniz algebras $2$-cocycle is
	$$
	\omega((x,y),(x',y'))=(1+a)xy'+(-1+a)x'y.
	$$	
	Now we can choose 
	$$
	G_0=\operatorname{SO}(2) \times \operatorname{SO}(2) \cong \lbrace (e^{ix},e^{iy}) \; \vert \; x,y \in \mathbb{R}\rbrace
	$$
	as a Lie group integrating $\mg_0$. In this case a Lie local rack integrating $\mg$ is $(G_0 \times \operatorname{SO}(2),\rhd)$ with multiplication
	$$
	(e^{ix},e^{iy},e^{iz}) \rhd (e^{ix'},e^{iy'},e^{iz'})=(e^{ix'},e^{iy'},e^{i(z'+\omega((\operatorname{log}(e^{ix}),\operatorname{log}(e^{iy})),(\operatorname{log}(e^{ix'}),\operatorname{log}(e^{iy'})))}),
	$$
	that is defined only for $(x,y),(x',y') \in [0,2\pi[ \times [0,2\pi[$, where we choose $[0,2\pi[$ as the domain of the principal value of the function $\operatorname{log}$. Thus the integration is not global.
\end{ex}

In order to show that \textit{Theorem} \ref{thm2} provides an effective tool for the construction of a global rack integrating a Leibniz algebra $\mg$ with $[\mg,\mg] \subseteq \operatorname{Z}(\mg)$, we can reformulate an example proposed by S. Covez in \cite{covez2013local}.

\begin{ex}
	Let $\mg=(\mathbb{R}^4,[-,-])$ with basis $\lbrace e_1,e_2,e_3,e_4\rbrace$ and nonzero brackets
	$$
	[e_1,e_1]=[e_1,e_2]=[e_2,e_2]=[e_3,e_3]=e_4, \; [e_2,e_1]=-e_4.
	$$
	It is easy to see that $\mg$ is a left Leibniz algebra with $[\mg,\mg]=\operatorname{Z}(\mg)=\mathbb{R}e_4$. We have that $\mg=\mg_0\oplus_\omega \R e_4$, where $\mg_0\cong \operatorname{Span}_\R\left\{e_1,e_2,e_3\right\}$, and the Leibniz $2$-cocycle is given by
	$$
	\omega(x,y)=\left[(x_1,x_2,x_3,0),(y_1,y_2,y_3,0)\right]=(0,0,0,x_1y_1+x_1y_2-x_2y_1+x_2y_2+x_3y_3).
	$$
	Thus, by \textit{Theorem} \ref{thm2}, a Lie global rack integrating $\mg$ is $\left(\mg_0\times_\omega \R e_4,\rhd\right)$ with multiplication given by
	$$
	(x_1,x_2,x_3,x_4)\rhd(y_1,y_2,y_3,y_4)=(y_1,y_2,y_3,y_4+\omega(x,y)).
	$$
\end{ex}

Now we can globally integrate all the indecomposable nilpotent real Leibniz algebras with one-dimensional commutator ideal classified in the previous sections. For the Heisenberg Leibniz algebras, we will obtain Lie racks that are "perturbations" of the conjugation of the Heisenberg Lie group $H_{2n+1}$. From now on we suppose that $A \in \operatorname{M}_n(\mathbb{R})$ is the companion matrix of the power of an irreducible monic polynomial $f(x) \in \mathbb{R}[x]$. Thus $A$ is a $n \times n$ Jordan block of eigenvalue $a \in \mathbb{R}$ or $A=J_R$, where $R \in \operatorname{M}_2(\mathbb{C})$ is the realification of some complex number $z=\alpha + i \beta$.

\begin{ex}
Let $\mg=\mathfrak{l}_{2n+1}^A$ and let $\left\{e_1,\ldots,e_n,f_1,\ldots,f_n,h\right\}$ be a basis of $\mg$. Then $\left[\mg,\mg\right]=\mathbb{R}h\subseteq \operatorname{Z}(\mg)$, thus we can use \textit{Theorem} \ref{thm2} to find the Lie global rack integrating $\mg$. The Leibniz bracket of $\mg$ is given by
$$
\left[(x_1,\ldots,x_n, y_1,\ldots, y_n, z),(x_1',\ldots,x_n', y_1',\ldots, y_n', z')\right]=\left(0,\ldots,0,\sum_{i,j=1}^{n}\left(\delta_{ij}+a_{ij}\right)x_iy_j'+\sum_{i,j=1}^{n}\left(-\delta_{ij}+a_{ij}\right)x_i'y_j\right),
$$ 
so we obtain a Lie rack $R_{2n+1}^{A}=(\mg_0\times_f \mathbb{R}h,\rhd)$ with multiplication 
\begin{align*}
	(x_1,\ldots,x_n,y_1,\ldots,y_n,z)&\rhd (x_1',\ldots,x_n',y_1',\ldots,y_n',z')=\\ &=\left(x_1',\ldots,x_n',y_1',\ldots,y_n',z'+\sum_{i,j=1}^{n}\left[\left(\delta_{ij}+a_{ij}\right)x_iy_j'+\left(-\delta_{ij}+a_{ij}\right)x_i'y_j\right]\right)
\end{align*}
and $\operatorname{T}_{(0,0)}R_{2n+1}^{A}=\mathfrak{l}_{2n+1}^A$.
\end{ex}

\begin{definition}
We define $(R_{2n+1}^{A},\rhd)$ as the \emph{Heisenberg rack}.
\end{definition}

We want to explicitly that the multiplication $\rhd$ in $R_{2n+1}^A$  is a $perturbation$ of the conjugation of the Heisenberg Lie group $H_{2n+1}$. To do this, we will use the canonical matrix representation 
$$
H_{2n+1}=\left\{
\begin{pmatrix}
	1 & x & z \\
	0 & I_{n} & y^t \\
	0 & 0 & 1
	\end{pmatrix} \Bigg\vert\;  x=(x_1,\ldots,x_n),y=(y_1,\ldots,y_n)\in\mathbb{R}^n, z\in\mathbb{R}
\right\}\leq \operatorname{GL}_{n+2}(\mathbb{R}).
$$

The conjugation formula for two matrices in $H_{2n+1}$ is given by
$$
\begin{pmatrix}
	1 & x & z \\
	0 & I_{n} & y^t \\
	0 & 0 & 1
\end{pmatrix}
\begin{pmatrix}
	1 & x' & z' \\
	0 & I_{n} & y'^t \\
	0 & 0 & 1
\end{pmatrix}
\begin{pmatrix}
	1 & x & z \\
	0 & I_{n} & y^t \\
	0 & 0 & 1
\end{pmatrix}^{-1}=
\begin{pmatrix}
	1 & x' & z'+\displaystyle\sum_{i=1}^{n}(x_iy'_i-y_ix_i')\\
	0 & I_{n} & y'^t \\
	0 & 0 & 1
\end{pmatrix}
$$
With the same representation, the multiplication $\rhd$ of $R_{2n+1}^A$ turns into 
$$
\begin{pmatrix}
	1 & x & z \\
	0 & I_{n} & y^t \\
	0 & 0 & 1
\end{pmatrix}\rhd
\begin{pmatrix}
	1 & x' & z' \\
	0 & I_{n} & y'^t \\
	0 & 0 & 1
\end{pmatrix}=
\begin{pmatrix}
	1 & x' & z'+\displaystyle\sum_{i,j=1}^{n}\left[\left(\delta_{ij}+a_{ij}\right)x_iy_j'+\left(\delta_{ij}-a_{ij}\right)x_i'y_j\right]\\
	0 & I_{n} & y'^t \\
	0 & 0 & 1
\end{pmatrix},
$$
hence for $A=0_{n\times n}$, it holds $R_{2n+1}^{0_{n\times n}}=\operatorname{Conj}(H_{2n+1})$.

\begin{ex}
	Let $\mg=\mathfrak{k}_{n}$ and let $\left\{e_1,\ldots,e_n,f_1,\ldots,f_n,h\right\}$ be a basis of $\mg$. Then the Leibniz bracket of $\mg$ is given by
	\begin{align*}
	[(x_1,\ldots,x_n, y_1,\ldots, y_n, z),&(x_1',\ldots,x_n', y_1',\ldots, y_n', z')]=\\
	&=\left(0,\ldots,0,x_1y_1'+\sum_{i=2}^{n}(x_iy_i'+x_iy_{i-1}'+
	x_{i-1}'y_{i-1}-x_{i}'y_{i-1})+x_n'y_n\right),
	\end{align*}
	so we obtain a Lie global rack $K_n=(\mg_0\times_f \mathbb{R}h,\rhd)$ with multiplication 
	\begin{align*}
		(x_1,\ldots,x_n,y_1,\ldots&,y_n,z)\rhd (x_1',\ldots,x_n',y_1',\ldots,y_n',z')=\\ &=\left(x_1',\ldots,x_n',y_1',\ldots,y_n',z'+x_1y_1'+\sum_{i=2}^{n}(x_iy_i'+x_iy_{i-1}'+
		x_{i-1}'y_{i-1}-x_{i}'y_{i-1})+x_n'y_n\right)
	\end{align*}
	and $\operatorname{T}_{(0,0)}K_{n}=\mathfrak{k}_{n}$.
\end{ex}

\begin{definition}
	We define $(K_n,\rhd)$ as the \emph{Kronecker rack}.
\end{definition}

\begin{ex}
	Let $\mg=\mathfrak{d}_{n}$ and let $\left\{e_1,\ldots,e_{2n+1},h\right\}$ be a basis of $\mg$. Then the Leibniz bracket of $\mg$ is given by
	\begin{align*}
		[(x_1,\ldots,x_{2n+1},z),&(x_1',\ldots,x_{2n+1}', z')]=(0,\ldots,0,\bar{z}),
	\end{align*}
    where
    $$
    \bar{z}=x_1x_{n+2}'+\sum_{i=2}^{n}x_i(x_{i+n}'+x_{i+n+1}')+x_{n+1}x_{2n+1}'+\sum_{i=n+2}^{2n+1}x_i(x_{i-n}'-x_{i-n-1}'),
    $$
	thus a Lie global rack integrating $\mg$ is $D_n=(\mg_0\times_f \mathbb{R}h,\rhd)$ with multiplication 
	\begin{align*}
		(x_1,\ldots,x_n,y_1,\ldots,y_n,z)\rhd (x_1',\ldots,x_n',y_1',\ldots,y_n',z')=(x_1',\ldots,y_1',\ldots,y_n',z'+\bar{z})
	\end{align*}
and $\operatorname{T}_{(0,0)}D_{n}=\mathfrak{d}_{n}$.
\end{ex}

\begin{definition}
	We call $(D_n,\rhd)$ the \emph{Dieudonné rack}.
\end{definition}

We want co conclude our paper with the following example. The realification of an indecomposable nilpotent Leibniz algebra with one-dimensional commutator ideal over the field $\mathbb{C}$ is a nilpotent real Leibniz algebra with two-dimensional commutator ideal. In the following example, we integrate the realification of the complex indecomposable Heisenberg Leibniz algebra $\mathfrak{l}_{2n+1}^{J_a}$, where $J_a \in \operatorname{M}_n(\mathbb{C})$ is the Jordan block of eigenvalue $a \in \mathbb{C}$. 

\begin{ex}
Let $\mathfrak{h}=\mathfrak{l}_{2n+1}^{J_a}$ and let $\left\{e_1,\ldots,e_n,f_1,\ldots,f_n,h\right\}$ be a basis of $\mathfrak{h}$ over $\mathbb{C}$. Then $dim_{\mathbb{R}}\mathfrak{h}=4n+2$ and $\left\{e_1,ie_1,\ldots,e_n,ie_n,f_1,if_1,\ldots,f_n,if_n,h,ih\right\}$ is a basis of $\mathfrak{h}$ over $\mathbb{R}$. For every $(x_1,\ldots,x_n, y_1,\ldots, y_n, z)$, $(x_1',\ldots,x_n', y_1',\ldots, y_n', z') \in \mathbb{C}^{2n+1}$, the Leibniz bracket of $\mathfrak{h}$ over $\mathbb{R}$ is given by
$$
\left[(x_1,\ldots,x_n, y_1,\ldots, y_n, z),(x_1',\ldots,x_n', y_1',\ldots, y_n', z')\right]=\left(0,\ldots,0,\Re(\bar{z}),\Im(\bar{z})\right),
$$ 
where $\Re(a+ib)=a, \Im(a+ib)=b$ and
$$
\bar{z}=\sum_{i=1}^{n}[(1+a)x_iy_i'+(-1+a)x_i'y_i]+\sum_{i=2}^{n}(x_iy_{i-1}'+x_i'y_{i-1}).
$$ Thus a Lie global rack integrating $(\mathfrak{h},[-,-])$ is $(\mathfrak{h}_0\times_f \operatorname{Span}_{\mathbb{R}}\lbrace h,ih \rbrace,\rhd)$ with multiplication 
\begin{align*}
	(x_1,\ldots,x_n,&y_1,\ldots,y_n,z)\rhd (x_1',\ldots,x_n',y_1',\ldots,y_n',z')=\\ &=\left(\Re(x_1'),\Im(x_1'),\ldots,\Re(x_n'),\Im(x_n'),\Re(y_1'),\Im(y_1'),\ldots,\Re(y_n'),\Im(y_n'),\Re(z'+\bar{z}), \Im(z'+\bar{z}) \right).
\end{align*}
\end{ex}

%

\printbibliography

\end{document}